\documentclass[a4paper,10pt]{article}

\usepackage{amsmath}
\usepackage{amsfonts}
\usepackage{amsthm}

\usepackage[english]{babel} 

\newcommand{\C}{\mathbb{C}}
\newcommand{\R}{\mathbb{R}}
\newcommand{\Z}{\mathbb{Z}}
\newcommand{\N}{\mathbb{N}}

\newtheorem{Theorem}{Theorem}[section]

\newtheorem{Corollary}[Theorem]{Corollary}
\newtheorem{Lemma}[Theorem]{Lemma}

\theoremstyle{remark}
\newtheorem{Remark}[Theorem]{Remark}
\theoremstyle{definition}

\def\H{{\rm \textbf{H}}}

\newcommand\Hy[1]{{\rm \textbf{H#1}}}
\renewcommand\Re{{\cal R}e}
\renewcommand\Im{{\cal I}m}

\title{Properties of periodic solutions near their oscillation threshold for a class of hyperbolic partial differential equations with localized nonlinearity}
\author{B. Ricaud}

\begin{document}

\maketitle
\begin{abstract}
The periodic solutions of a type of nonlinear hyperbolic partial differential equations with a localized nonlinearity are investigated. For instance, these equations are known to describe several acoustical systems with fluid-structure interaction. It also encompasses particular types of delay differential equations. These systems undergo a bifurcation with the appearance of a small amplitude periodic regime. Assuming a certain regularity of the oscillating solution, several of its properties around the bifurcation are given: bifurcation point, dependence of both the amplitude and period with respect to the bifurcation parameter, and law of decrease of the Fourier series components. All the properties of the standard Hopf bifurcation in the non-hyperbolic case are retrieved.
In addition, this study is based on a Fourier domain analysis and the harmonic balance method has been extended to the class of infinite dimensional problems hereby considered. Estimates on the errors made if the Fourier series is truncated are provided.
\end{abstract}

\section{Introduction}

\subsection{Physical motivations}

The starting point of the present work is due to physical motivations. The first objective is to justify a thirty-year-old conjecture assumed in order to obtain information on the bifurcations encountered in woodwind musical instruments. 
It is now established~\cite{FR,KC} that woodwind instruments are nonlinear systems having a transition between a steady state and an oscillating regime, for a certain value of the pressure in the mouth of the musician. These instruments possess a fluid-structure interaction which implies the presence of partial differential equation (PDE) along with a nonlinearity. Analytical calculations have been done for small oscillations around the bifurcation, giving valuable information on the oscillation with respect to the parameters of the instrument. But these calculations are possible only if a conjecture, first presented by Worman~\cite{W}, is assumed. Suppose the instrument emits a periodic sound of small amplitude, and the acoustic pressure $p_0$ inside the mouthpiece is oscillating with period $\tau$. Let $p_a$ be the atmospheric pressure (or any other suitable pressure scale) and define the dimensionless quantity $p=p_0/p_a$. Let $P(\omega_q)$ be the Fourier component of $p$ associated with the angular frequency $\omega_q=2q\pi/\tau$. Worman's conjecture states:
$$|P(\omega_q)|\le|P(\omega_1)|^{|q|}, \text{ for all } q\in\Z^*,
$$
where, since it is of small amplitude, $|P(\omega_1)|<1$. This key hypothesis is not obvious and arguments to justify it have been given in the theoretical work of Grand et al~\cite{GGL}. However, the Fourier series has been truncated in the calculations of this latter work and the control of the remaining tail was not investigated. The first aim of the present study is to show that the weaker assumption of a twice continuously differentiable solution for the model of woodwind instruments implies Worman's conjecture. This former regularity assumption, allowing to control the tail of the Fourier series, is quite natural from a physical point of view.

In addition to this first motivation, the hypotheses on the equations are weak enough to hold for a broader class of physical systems. 
The present approach is especially suitable for vibrating systems described by equations in the frequency domain with the formalism of transfer functions. This study proposes an extension of the harmonic balance method, widespread in engineering~\cite{MC,AL}, and of the underlying Galerkin procedure, to infinite dimensional systems governed by partial differential equations with a localized nonlinearity. This generalization holds for the case of small amplitude solutions which are twice continuously differentiable with respect to time.
The present study justifies in a mathematical way the frequency domain expression, the truncation of the Fourier series, and the analysis around the bifurcation threshold.

\subsection{Mathematical framework} 

Properties of the solutions of a nonlinear system around a bifurcation point can be obtained by general techniques and theorems such as the reduction to the center manifold and normal forms~\cite{HI},~\cite{DFKM} or the Hopf bifurcation theorem~\cite{CR}.
Nevertheless, in the case of PDEs involving the wave equation, which makes them hyperbolic, the hypotheses required for the use of these powerful theorems may not be satisfied. In such systems, the norm of the resolvent associated to the linearized operator in a neighbourhood of a bifurcation may not decrease when the spectral parameter goes to infinity. This is the case for woodwind musical instruments.
The present study shows that although the hypotheses of the Hopf theorem are not satisfied, several of the results given by this theorem can be retrieved, provided the periodic solution is of small amplitude and twice continuously differentiable. The existence of such a solution is an hypothesis of the present work. Notice however that several studies show that nonlinear systems with PDEs possess such regular solutions, for example the case of the wave equation with nonlinear terms and damping, see e.g.~\cite{GS,CLT,Co}. 

After having settled the hypotheses satisfied by the systems under study, a correspondence between expressions of the model in the time and the frequency domain is established in lemma~\ref{lemmetf}. Then theorems~\ref{Thbif} gives required conditions for an oscillation to appear and theorems~\ref{Thbif} and~\ref{ThP1}, with corollary~\ref{corolomega} give the dependence of both the Fourier harmonics and the frequency on the bifurcation parameter, near its birth point. As expected, it is similar to the finite dimensional case and what is given by the finite dimensional Hopf bifurcation theorem. In addition, this study gives also a valuable estimation of the law of decrease of the harmonics of the solution with respect to the multiple of the fundamental frequency in theorem~\ref{Th1}. This justifies the process of truncation used in the harmonic balance method and gives an error term on it.
The last section comes back to the physical motivations and is devoted to an application in the field of musical acoustics. It introduces the equations modeling a reed woodwind instrument and shows how the mathematical results can be applied to this example.

\section{The system}\label{system}

Let $t\in T\subset\R$ denote the time and $X\subset\R^n$ the space of configurations, for $n\in\N$. Let ${\cal X}=L^2(T\times X)$ be the separable Hilbert space of square integrable real-valued functions on the domain $T\times X$. Only time periodic solutions are investigated and $T\times X$ is supposed to be a compact domain of $\R^{n+1}$. The nonlinearity takes place at a particular location $x_0\in X$ in space and the infinite dimensional dynamical system under study is described formally by 
\begin{align}\label{eqdep}
Ap&=\delta(x-x_0)B{\cal G}(\gamma,p),
\end{align}
where $p\in{\cal X}$ is the state of the system, the real number $\gamma$ is to be called the bifurcation parameter, $A$ is a linear partial differential operator. The right hand side of~\eqref{eqdep} models the localized nonlinearity with ${\cal G}$ being a nonlinear function of $p$, $B$ a linear operator acting on ${\cal G}$ and the symbol $\delta$ being the Dirac distribution.
As a consequence,~\eqref{eqdep} must be understood in the distribution sense. In addition, since $T\times X$ is compact, some linear boundary conditions are associated to~\eqref{eqdep}. In a woodwind musical instrument, for instance a clarinet, the function $p$ represents the air pressure inside the bore (resonator) and $Ap$ is the wave equation (which may contain dissipative terms) in this domain of finite length. The position $x_0$ is the top of the instrument, where the tip of the reed is located. This flexible part oscillates when the musician plays and influence the air flow entering the instrument. The parameter $\gamma$ is proportional to the pressure inside the mouth. The function ${\cal G}$ is the air flow through the clarinet, it depends in a nonlinear manner on the pressure difference between the player's mouth ($\gamma$) and the resonator ($p$). Here, $B$ is the time derivative operator. It is necessary to blow hard enough in order to obtain a sound, in other words the dynamical system encounters a bifurcation for some particular value of $\gamma$. It is important to understand how the oscillating frequency of the sound depends on the parameters of the instrument and the instrumentalist. Form example, obtaining this frequency will help understand how a beginner should act to avoid a high frequency squealing oscillation. Since $A$ is a differential operator one can always find a solution $p$, well defined as a distribution (see e.g. Schwartz~\cite{Sch} section VII-10 and the examples therein). 
Let ${\widetilde p}$ be the Green function (distribution), solution of $A\widetilde{p}=\delta(x-\xi)\delta(t)$ and the boundary conditions. Thus $p$ is given by the convolution $p=\widetilde{p}*\delta(x-x_0)B{\cal G}(\gamma,p)$. Let $p_0\in L^2(T)$ be the value of $p$ at $x_0$.
Let us assume that $B$ and ${\cal G}$ are decoupled from the space variable i.e. they are such that:
$$\int_X\delta(x-x_0)B{\cal G}(\gamma,p)dx=B{\cal G}(\gamma,p_0).
$$
It is the case for the clarinet where the air flow ${\cal G}$ at $x_0$ depends only on the pressure in the player's mouth ($\gamma$), the internal pressure $p$ at $x_0$ and on the size of the opening at this point.
Let us further assume that $A$ and ${\cal G}$ possess the properties allowing to write:
\begin{align}
p(x,t)&=\int_{X\times T}\widetilde{p}(x,t-\tau;\xi)\delta(\xi-x_0)\left.B{\cal G}(\gamma,p)\right|_{\xi,\tau}d\xi d\tau\nonumber\\
&=\int_{ T}\widetilde{p}(x,t-\tau;x_0)\left.B{\cal G}(\gamma,p_0)\right|_{\tau}d\tau.
\end{align}
As a consequence of the above equation, there exists an operator $\widetilde{Z}:L^2(T)\to{\cal X}$ such that:
\begin{align}\label{eqPZG}
 p=\widetilde{Z}{\cal G}(\gamma,p_0).
\end{align}
Notice that the action of $\widetilde{Z}$ implies a convolution with respect to time, which becomes a multiplication by Fourier transform. Expressing~\eqref{eqPZG} at $x=x_0$, one can introduce the operator $Z=\widetilde{Z}|_{x_0}$, where $Z:L^2(T)\to L^2(T)$, and the study is hence reduced to the investigation of the oscillating solutions belonging to $L^2(T)$ of:
\begin{align}\label{eqreel}
p_0=Z{\cal G}(\gamma,p_0).
\end{align}
Equation~\eqref{eqdep} is a formal expression whereas~\eqref{eqreel} is rigorously defined in the following and is the starting point of the study. The woodwind instruments have been modelled by physicists in the framework of control theory where the equations are written in term of transfer functions. These latter quantities are the Laplace transform of some differential equations, and differential operators become multiplication operators under this transformation. The Fourier transform of the operator $Z$ in~\eqref{eqreel} is a multiplicative operator and~\eqref{eqreel} follows the control theory formalism. Thus, the present study is oriented toward control theory, however, with appropriate assumptions, it also applies to the usual description with partial differential equations.

The nonlinear vector valued function ${\cal G}:(\R\times L^2(T))\to L^2(T)$ has the particular shape:
\begin{align}\label{eqdep2}
{\cal G}(\gamma,p_0)&=Lp_0+\sum_{n=2}^{\infty}R_n(\gamma,p_0),
\end{align}
where the linear operator $L$ depends on $\gamma$ and each $R_n$ corresponds to a monomial of degree $n$ in $p_0$. More precisely, for each $n\ge2$, let $Z_{n,1}$, $Z_{n,2},\cdots$ $Z_{n,n}$ be bounded linear operators from $L^2(T)$ to itself. Each operator may depend on $\gamma$. It is assumed that the nonlinear terms can be written as
\begin{align}
R_n(\gamma,p_0)=(Z_{n,1}p_0)(Z_{n,2}p_0)\cdots (Z_{n,n}p_0).
\end{align}
This expression of the nonlinearity allows to cover the case of woodwind musical instruments as it will be shown in section~\ref{woodwind} as well as different types of delay differential equations.
\begin{Remark}\label{rmk1}
Let some dynamical system be described by the standard notation:
\begin{equation}\label{eqstd}
\frac{dp_0}{dt}=L_2p_0+g(\gamma,p_0)
\end{equation}
where $L_2$ is the linear part and $g$ the nonlinear mapping. See~\cite{GH} for examples of assumptions on theses objects, \cite{HI} for the infinite dimensional case, or~\cite{CR} for the case of the Hopf bifurcation. The connection with this study is made by writing formally $Z^{-1}(1-ZL)=L_2-d/dt$. By Fourier transform the time derivative operator becomes the product by some  complex number $i\omega$, $\omega\in\R$. Looking for a pure imaginary eigenvalue of $L_2$ is the same problem as searching for a zero eigenvalue of $Z^{-1}(1-ZL)$.
In order to prove the existence of a Hopf bifurcation, the norm of the resolvent $(L_2\pm i\omega)^{-1}$ must be bounded by some constant$/|\omega|$ when the spectral parameter $\pm i\omega$ tends to complex infinity (see~\cite{HI}), a property which does not arise in hyperbolic PDEs. Here the inverse of $Z^{-1}(1-ZL)$ is only required to be bounded for each $\gamma$ in a neighborhood of the bifurcation point, except at this critical point. However, this decrease of the norm of the resolvent is replaced here by the assumption of a twice continuously differentiable periodic solution.
\end{Remark}

Let $L^{\infty}(T)$ be the space of bounded functions on $T$ and let ${\cal B}_{\varepsilon}\subset L^{\infty}(T)$ be the ball of radius $\varepsilon>0$ centered at zero. Let $C^2(T)$ be the set of functions twice continuously differentiable with respect to $t\in T$ and define
$$\Lambda={\cal B}_{\varepsilon}\cap C^2(T).$$
The solution of~\eqref{eqreel} under investigation satisfies the following hypothesis:
\begin{itemize}
\item[\Hy1] For all $\varepsilon\in(0,1)$ there exists a non-empty, open, simply connected set $\Omega_1\subset\R$, where for all $\gamma\in\Omega_1$ the system has a unique non trivial real-valued periodic solution belonging to $\Lambda$, continuous with respect to $\gamma$. Its period will be denoted by $\tau=\tau(\gamma)$. There exists a $\gamma_0\in\R$, with $\gamma_0\in\Omega$ the closure of $\Omega_1$, such that the periodic solution tends to zero when $\gamma$ tends to $\gamma_0$. The point $\gamma_0$ will be called the critical point or the bifurcation point.
\end{itemize}

As said in the introduction, and stated in \Hy1, the regularity of the solution is an hypothesis. Nevertheless, several publications have shown regularity properties of the solutions of nonlinear hyperbolic PDEs. For instance~\cite{GS} shows that a twice continuously differentiable solution with respect to time exists for a model of the wave equation with damping and source term, this solution being a distribution with respect to the space variable. In~\cite{CLT}, for a slightly different model, the solution is twice continuously differentiable with respect to time and, for all time fixed in a bounded interval, it belongs to a Sobolev space. The periodic solutions of the nonlinear wave equation are even more regular for some class of nonlinear terms as shown in~\cite{Co} and references therein. To the knowledge of the author, no such results have been settled in the case of a nonlinear function localized at one point in space. However, the model described by~\eqref{eqdep}, where $A$ represents the wave equation in a tube with damping, has been investigated theoretically~\cite{SKVG,RGKSV} and numerically~\cite{GKN,CV} by physicists, giving results compatible with this regularity hypothesis, and motivating the present approach.

Let us introduce the notation
$$J_{\gamma}=Z^{-1}(1-ZL),
$$
and assume:
\begin{itemize}
\item[\Hy2] 
\begin{itemize} 
\item[1)] $J_{\gamma}$ is a closed operator acting in $L^2(T)$. It is diagonalizable and its eigenvectors consist of the vectors of the Fourier basis of $L^2(T)$. 
\item[2)] Let $Q_e$, $Q_{-e}$ be the orthogonal projections associated to the vectors $\psi_e$, $\psi_{-e}$, respectively, of the Fourier basis. $J_{\gamma}$ is such that
$$Q_eJ_{\gamma}Q_e=\beta_e(\gamma)Q_e,\quad Q_{-e}J_{\gamma}Q_{-e}=\beta_{-e}(\gamma)Q_{-e}
$$
where $\beta_e(\gamma)$, $\beta_{-e}(\gamma)$ are two complex conjugate eigenvalues of $J_{\gamma}$ such that $\beta_e(\gamma_0)=0$. Moreover, for all $\gamma\in\Omega$, the operator $QJ_{\gamma}Q$, where $Q=(1-(Q_e+Q_{-e}))$, possesses a bounded inverse.
\end{itemize}
\end{itemize}
On the nonlinear part, the hypotheses are:
\begin{itemize}
\item[\Hy3] For all $\gamma\in\Omega$ and $n,m\in\N$, $n\ge2$, $m\in[1,n]$, $Z_{n,m}(\gamma):L^2(T)\to L^2(T)$ is a bounded linear operator. Moreover, as for $J_{\gamma}$, the Fourier basis is its eigenbasis. For all $n\ge2$, $\gamma\in\Omega$ and $p\in\Lambda$, by the H\"older inequality, the quantity $R_n(\gamma,p_0)$ belongs to $L^2(T)$, since the $\{Z_{n,m}\}$ are bounded and $\|p_0\|_{\infty}\le\varepsilon$, where $\|\cdot\|_{\infty}$ denotes the norm associated to $L^{\infty}(T)$.
Eventually, it is supposed that for all $\gamma\in\Omega$ and $p_0\in\Lambda$, the series of nonlinear terms is convergent :
$$\lim_{M\to\infty}\sum_{n=2}^M R_n(\gamma,p_0)=\overline{R}(\gamma,p_0)\in L^2(T).
$$
\end{itemize}
\begin{Remark}
If for all $\gamma\in\Omega$, ${\cal G}$ is an analytic function in its second variable, then it satisfies the above hypothesis: all the $Z_{n,m}$ are identity operators except when $m=1$, where $Z_{n,1}$ is the term of the Taylor expansion $Z_{n,1}={\cal G}^{(n)}(\gamma,0)/n!$, with ${\cal G}^{(n)}$ denoting the $n$th derivative.
\end{Remark}

Since the solution is periodic in time, it is natural to work in the Fourier space. Let $S^{\tau}$ be the circle of perimeter $\tau$, then $ T=S^{\tau}$. Let us define $\ell^2$, the Hilbert space of sequences which are square summable. Let ${\cal F}$ denotes the Fourier transform operator from $L^2(T)$ to $\ell^2$. The Fourier transform of $p_0$ will be denoted $\widehat{p_0}={\cal F} p_0\in\ell^2$ or, using the physicists's convention, by a capital letter $P=\widehat{p_0}$. For all $q\in\Z$ and $\omega_q=\frac{2\pi q}{\tau}$, the Fourier transform is defined by :
\begin{align*}
P_q=\widehat{p}(\omega_q)=({\cal F}p_0)(\omega_q)=\frac{1}{\tau}\int_{S^{\tau}}p_0(t)e^{-i\omega_q t}dt,
\end{align*}
and the inverse transform at time $t\in S^{\tau}$ is:
\begin{align*}
({\cal F}^{-1}P)(t)=\sum_{q\in\Z}P_qe^{i\omega_q t}.
\end{align*}
Using integration by part in the expression of the Fourier transform, the hypothesis of twice continuously differentiable solutions of \Hy1 implies the existence of $a_0>0$ such that:
\begin{align}\label{Pdecr}
\forall q\neq0,\quad|P_q|\le \frac{a_0}{|q|^2}.
\end{align}
Moreover, from \Hy1 and the definition of the Fourier transform, the small amplitude condition gives for all $\gamma\in\Omega_0$:
\begin{equation}\label{boundPq1}
\sup_{q\in\Z}|P_q|\le \varepsilon.
\end{equation}

\section{Expression in the frequency domain}
Many physical systems having an oscillating behavior, for instance in mechanics, acoustics or electrical engineering are described by relations taking place in the frequency domain. This is the formalism of transfer functions. Moreover, investigating oscillations turns out to be extremely efficient when using a frequency domain approach. In finite dimension, the harmonic balance method or the describing function method are such techniques. In infinite dimension, these techniques have not been proved to hold. This section, with lemma~\ref{lemmetf}, shows how the infinite dimensional system under study can be written in the frequency domain. It is the first step toward its analysis using frequency domain techniques.

The key point is to express the nonlinear relationship in the Fourier basis. It will become a combination of convolutions. Let $\widehat{g},\widehat{h}\in \ell^2$ be such that their harmonics decrease like in~\eqref{Pdecr}, in consequence they belong to $\ell^1$, the space of summable sequences. Thus, the convolution $\widehat{g}*\widehat{h}$ defined at $q\in\Z$ by :
$$(\widehat{g}*\widehat{h})(q)=\sum_{n\in\Z}\widehat{h}_{q-n}\widehat{g}_n,
$$
is an $\ell^1$ function by Young's inequality, see e.g.~\cite[Sec. IX.4]{RS2}.
For each $Z_{n,m}$ defined in \Hy3, let us introduce the notation:
\begin{equation*}
 \widehat{Z_{n,m}}={\cal F}Z_{n,m}{\cal F}^{-1}.
\end{equation*}
From \Hy3, since each $Z_{n,m}$ is diagonal in the Fourier domain, one has for the value of $P$ associated to the vector $\psi_q$ of the Fourier basis:
\begin{align*}
({\cal F}Z_{n,m}p)(q)=\widehat{Z_{n,m}}(q)P_q
\end{align*}
where $\widehat{Z_{n,m}}(q)$ is a complex number.

More generally, let $p,u_1,u_2,\cdots,u_{N-1}$ be $N\ge2$ real periodic functions depending on time $t\in S^{\tau}$. Let these functions describe the state of some physical system under study. Let $P,U^{(1)},U^{(2)},\cdots,U^{(N-1)}$ be the Fourier transforms of $p,u_{1},u_{2},\cdots,u_{N-1}$ respectively. 
Suppose they are related together through $N-1$ linear relations, in the frequency domain, involving transfer functions:
\begin{align}\label{UZP1}
U^{(m)}=\widehat{Z_m}P, 
\end{align}
where the $\{\widehat{Z_m}\}$ are bounded operators. All may depend on the bifurcation parameter. Let us suppose there is also one nonlinear relation, often expressed in the time domain:
\begin{align}\label{nonlinex}
{\cal N}(\gamma,p,u_1,u_2,\cdots,u_{N-1})=0.
\end{align}
If this relationship involves polynomial terms between the variables, such as $p^2$, $pu_1$ or $u_1^2u_2$, the following lemma can be used to express this relation in the frequency domain, with convolution products. For example, a cubic term involving three functions $p$, $u_1=Z_1p$, $u_2=Z_2p$, with $p$ satisfying \Hy1 and $Z_1,Z_2$ satisfying \Hy3 would give :
$$\widehat{pu_1u_2}=\widehat{p}*\widehat{u_1}*\widehat{u_2}=\widehat{p}*(\widehat{Z_1}\widehat{p})*(\widehat{Z_2}\widehat{p}).
$$
And for all $q\in\Z$:
$$\left\{\widehat{p}*(\widehat{Z_1}\widehat{p})*(\widehat{Z_2}\widehat{p})\right\}(\omega_q)=\sum_{k_1\in\Z}\sum_{k_2\in\Z}\widehat{Z_1}(\omega_{k_1}-\omega_{k_2})\widehat{Z_2}(\omega_{k_2})P_{q-k_1}P_{k_1-k_2}P_{k_2}.
$$
The whole relation~\eqref{nonlinex} can be stated under a set of harmonic balance equations, in the frequency domain, given by~\eqref{eqfreq}.

\begin{Lemma}\label{lemmetf}
Suppose $p_0\in\Lambda$ and $P$ denotes its Fourier transform. Under the hypotheses \Hy2, \Hy3, for all $\gamma\in\Omega$, all $n\in\N$, $n\ge2$,
$${\cal F}R_n(\gamma,p_0)=S(g_n,n,P),
$$
where $S(g_n,n,P)\in \ell^1$. For all $q\in\Z$ the elements of $S(g_n,n,P)$ are defined by :
\begin{equation}\label{defS}
S_q(g_n,n,P)=\sum_{k_1\in\Z}\sum_{k_2\in\Z}\cdots\sum_{k_{n-1}\in\Z}g_n(q,k_1,\cdots,k_{n-1})P_{q-k_1}P_{k_1-k_2}\cdots P_{k_{n-1}},
\end{equation}
where for all $n$, $g_n$ is a complex-valued function of $n$ variables depending on combinations of $Z_{n,m}$:
\begin{align*}
g_n(q,k_1,\cdots,k_{n-1})&=\widehat{Z_{n,1}}(q-k_1)\widehat{Z_{n,2}}(k_1-k_2)\cdots\widehat{Z_{n,n}}(k_{n-1}).
\end{align*}
Moreover, the infinite sum gives:
$${\cal F}\sum_{n=2}^{\infty}R_n(\gamma,p_0)=\sum_{n=2}^{\infty}S(g_n,n,P)\in\ell^2.
$$
In consequence, solving the system described by~\eqref{eqreel} is equivalent to solving:
\begin{align}\label{eqfreq}
\widehat{J_{\gamma}}P=\sum_{n=2}^{\infty}S(g_n,n,P).
\end{align}
\end{Lemma}
Notice that although not specified, the convolution $S$ depends on $\gamma$.
\begin{proof}[Proof of lemma~\ref{lemmetf}]
Let $G^{(1)},G^{(2)},\cdots,G^{(n+1)}\in \ell^1$, $n\in\N$, $n\ge1$, then the $n$-product of convolution $G^{(1)}*G^{(2)}*\cdots*G^{(n+1)}$ is an $\ell^1$ function, by Young inequality. 
Suppose $g,h\in L^2(T)$ such that their respective Fourier transform $G,H$ are $\ell^1$ functions. Then Fubini's theorem can be used to get:
\begin{align}\label{mulconv}
\sum_{q\in\Z} \sum_{k\in\Z}H_{q-k}G_ke^{i\omega_qt}&= \sum_{q\in\Z} \sum_{k\in\Z}H_{q-k}e^{i\omega_{q-k}t}G_ke^{i\omega_kt}=\sum_{k\in\Z} \sum_{q\in\Z}H_{q-k}e^{i\omega_{q-k}t}G_ke^{i\omega_kt}\nonumber\\
&=\sum_{k\in\Z}h(t)G_ke^{i\omega_kt}=h(t)g(t).
\end{align}
Notice that the functions $h$, $g$ are continuous by the Riemann-Lebesgue lemma. Reciprocally, for $h$, $g\in L^2(T)$ twice continuously differentiable, $H$ and $G$ are $\ell^1$ functions (see explanation of~\eqref{Pdecr}) and this equality holds.
In consequence, Eq.~\eqref{mulconv} shows the correspondence between multiplication in the time domain and convolution in the frequency domain of a finite number of products, in the case of twice continuously differentiable functions on the compact domain $T$.
Since \Hy3 holds, all the $\{Z_{n,m}\}$ are bounded and for all $n\ge2$, $g_n$ is bounded. The bound may depend on $n$.
Using H\"older inequality, one gets:
$$\|R_n(\gamma,p)\|_{L^2(T)}\le \|Z_{n,1}\cdots Z_{n,n}\|\cdot(\|p\|_{\infty})^{n-1}\cdot\|p\|_{L^2(T)},
$$
which shows that $R_n(\gamma,p)$ belongs to $L^2(T)$. The symbol $\|\cdot\|$ denotes the operator norm.
Suppose $S(1,n,P)\in\ell^1$, then
$$S(1,n+1,P)=P*S(1,n,P)\in\ell^1
$$
and since $S(1,2,P)\in\ell^1$, by induction, for all $n\ge2$, $S(1,n,P)\in\ell^1$.
Then, one has:
$$\|S(g_n,n,P)\|_{\ell^1}\le \sup|g_n|\cdot\|S(1,n,P)\|_{\ell^1}.
$$
In consequence,~\eqref{mulconv} shows that for $n\ge2$ fixed, $S(g_n,n,P)$ is the Fourier transform of $R_n(\gamma,p)$. Let us study the limit $n\to\infty$. Since the Fourier transform is bounded, and even unitary one has
$$\overline{S}(P)={\cal F}\overline{R}(\gamma,p)\in\ell^2,
$$
and the quantity:
$$\left\|\overline{S}(P)-\sum_{n=2}^MS(g_n,n,P)\right\|_{\ell^2}\le \left\|\overline{R}(\gamma,p)-\sum_{n=2}^MR_n(\gamma,p)\right\|_{L^2(T)}
$$
tends to zero as $M$ tends to infinity from \Hy3.
\end{proof}

\section{Preliminary results and properties at threshold}\label{threshold}
This section demonstrates how the method based on the frequency domain approach can be used to investigate a bifurcation in a nonlinear system. It shows how to handle the hypotheses (here \H1, \H2 and \Hy3) and the Fourier harmonics in order to control infinite sums and convolutions.
The first result, stated in the theorem which follows, concerns the determination of the birth point of the oscillation and its connection to particular eigenvalues of $J_{\gamma}$.
The second result gives an estimation of the harmonics and shows the importance of fundamental component. This serves as an introduction to Th.~\ref{Th1} of section~\ref{decrease} which will go further in this estimation.
This section settles many important intermediate results which will be used in the following parts of the study.
\begin{Theorem}\label{Thbif}
Suppose \Hy1, \Hy2 and \Hy3 hold. Then the bifurcation point $\gamma_0$ where the oscillating solution starts and the frequency $\omega_e$ of oscillation at this point are given by the characteristic equation:
$$\beta_e(\gamma_0)=0.
$$
Moreover, for all $\gamma\in\Omega$ where $\varepsilon$ is small enough: the fundamental component of the oscillation $P_e$ is such that $|P_e|=\varepsilon$, and there exists $c_0>0$ for the harmonics such that for all $q\in\{ne\}_{n\in\Z}\backslash\left\{e,-e\right\}$ , $|P_q|\le c_0\varepsilon^2$.
\end{Theorem}
\begin{Remark}
The quantities $e$, $\varepsilon$ and $\gamma$ are dependent of each other and this dependence will be made explicit in the section~\ref{stable}.
\end{Remark}

\begin{Remark}
In the case of the Hopf bifurcation, the condition of two purely imaginary eigenvalues of the Jacobian operator at the bifurcation determines $\gamma_0$ and the oscillating frequency $\omega_e$ at this point. In the present case, the result is equivalent but with the condition: two complex conjugate eigenvalues of $J_{\gamma}$ reach zero at $\gamma_0$ (see remark~\ref{rmk1}).
\end{Remark}

Before stating the proof of this theorem, several lemmas giving bounds on the different convolution products encountered are presented. The aim is to establish a bound on the right hand side of~\eqref{eqfreq} and show that this bound is smaller than the initial one on $|P_q|$ (Eqs.~\eqref{Pdecr} and~\eqref{boundPq1}), for particular values of $q$. The first lemma starts with an estimate on the convolution product, which is present in all the $S$ functions defined in~\eqref{defS}.
\begin{Lemma}\label{lemmeS1}
Let $a_1,a_2$ be positive constants and let $\lambda\in(0,1)$, $\theta\ge0$ and $\alpha>0$. Let $f,g:\Z\to\C$ be two bounded functions such that 
$$\forall q\in\Z,\quad |f(q)|\le\lambda,\quad |g(q)|\le\lambda^{1+\theta},
$$
and $\forall q\in\Z^*$,
$$ |f(q)|\le\frac{a_1}{|q|^{2}},\quad |g(q)|\le\frac{a_2}{|q|^{1+\alpha}}.
$$
Then, there exists $c_1>0$ and a positive continuous decreasing function $c_2:\R_+^*\to\R_+$ such that
\begin{itemize}
\item[i)]
$$\forall q\in\Z,\quad|(f*g)(q)|\le a_1^{3/4}c_1\cdot \lambda^{r},
$$
with
$$ r=1+\frac{1}{4}+\theta,
$$
\item[ii)]
$$\forall q\in\Z^*,\quad |(f*g)(q)|\le 3 \max\{a_2\lambda, a_1\lambda^{1+\theta} ,c_2(\alpha)\lambda^{\alpha/8}\} \cdot\frac{1}{|q|^{\sigma}},
$$
with $\sigma=1+\frac{\alpha}{4}$.
\end{itemize}
The function $c_2$ tends to infinity when $\alpha$ tends to zero.
\end{Lemma}
\begin{proof}
Let us introduce the notation $f_q=f(q)$ for $f:\Z\to\C$. Let $f,g$ be two functions satisfying the hypothesis. For all $q\in\Z$,
\begin{align*}
|\sum_{k\in\Z} f_{q-k}g_k|&\le 2\lambda^{2+\theta}+\sup_{k\in\Z^*,k\neq q}|f_{q-k}|^{1/4}|g_k|\sum_{k\in\Z^*,k\neq q} |f_{q-k}|^{3/4} \\
&\le 2\lambda^{2+\theta}+  \lambda^{1+1/4+\theta}a_1^{3/4}\sum_{k\in\Z^*,k\neq q} \frac{1}{|q-k|^{3/2}}\\
&\le a_1^{3/4}c_1 \lambda^{1+1/4+\theta},
\end{align*}
where $c_1>0$.

For $ii)$, for $q\in\Z^*$ and $\eta\in[0,\alpha/2)$, the following estimate can be written:
\begin{align*}
|\sum_{k\in\Z} f_{q-k}g_k|&\le \frac{a_2\lambda}{|q|^{1+\alpha}}+\frac{a_1\lambda^{1+\theta}}{|q|^{2}}+\lambda^{\eta}\sum_{k\in\Z^*,k\neq q} |f_{q-k}|^{1-\eta}|g_k| \\
&\le \frac{a_2\lambda}{|q|^{1+\alpha}}+\frac{a_1\lambda^{1+\theta}}{|q|^{2}}+\lambda^{\eta} \sum_{k\in\Z^*,k\neq q} \frac{a_1^{1-\eta}a_2}{|q-k|^{2(1-\eta)}|k|^{1+\alpha}}.
\end{align*}
For any $\beta\in[0,1+\alpha/2-\eta)$ one has the bound:
\begin{align*}
\sum_{k\in\Z^*,k\neq q} \frac{1}{|q-k|^{2(1-\eta)}|k|^{1+\alpha}}&\le \sup_{k\in\Z^*,k\neq q }\frac{1}{(|q-k||k|)^{\beta}}\sum_{k\in\Z^*,k\neq q} \frac{1}{|q-k|^{2(1-\eta)-\beta}|k|^{1+\alpha-\beta}}\\
&\le \widetilde{c}(\alpha,\eta,\beta)\sup_{k\in\Z^*,k\neq q }\frac{1}{(|q-k||k|)^{\beta}},\qquad \text{with }\widetilde{c}(\alpha,\eta,\beta)>0.
\end{align*}
The function $\widetilde{c}$ is given by applying Young's inequality under the condition $(3+\alpha-2\eta-2\beta>1)$. This condition is satisfied if $\beta<1+\alpha/2-\eta$ and $\alpha/2-\eta>-1$. The function $\widetilde{c}$ tends to infinity when $\alpha/2-\eta$ tends to $(-1)$.
The function $h_q(k)=|q-k||k|$ possesses two equal minima at $k=0$ and $k=q$ (and a local maximum at $k=q/2$). Therefore, for $k,q\in\Z$, $k\neq0$, $k\neq q$ and $|q|>1$, the minimum of the function $h_q$ is:
$$\min_{k\in\Z^*,k\neq q}h_q(k)=||q|-1|.
$$
Remark that for $q=\pm1$, the minimum is $1$ as $k\neq q$. For all $q$ such that $|q|>1$, one has $|q|/2\le||q|-1|$ and
\begin{align*}
\sum_{k\in\Z^*,k\neq q} \frac{1}{|q-k|^{2(1-\eta))}|k|^{1+\alpha}} &\le\frac{2^{\beta}\widetilde{c}(\alpha,\eta,\beta)}{|q|^{\beta}},\qquad q\neq0.
\end{align*}
For the case $q=\pm1$, the bound is simply $\widetilde{c}(\alpha,\eta,\beta)$. Setting $\eta=\alpha/8$ and $\beta=1+\alpha/4$ concludes the proof.
\end{proof}

The next lemma uses the previous one to state upper bounds on the value of the function $S$ defined in~\eqref{defS}.
\begin{Lemma}\label{lemmeSn1}
Suppose there exists $a_0>0$, $N\ge2$ and a small $\varepsilon=\varepsilon(N)\in(0,1)$ such that for all $q\in\Z$, 
\begin{itemize} 
\item[i)] $|P_q|\le\varepsilon$, 
\item[ii)]for all $q\in\Z^*$, $|P_q|\le\frac{a_0}{|q|^2}$.
\end{itemize}
Then for all $n\in\N,n\in[2,N]$, $q\in\Z$:
\begin{align}\label{boundS}
|S_q(1,n,P)|\le \varepsilon^{1+\theta_n},\qquad\theta_n=\frac{1}{8}(n-1).
\end{align}
\end{Lemma}
\begin{proof}
A direct application of lemma~\ref{lemmeS1} gives for the case $n=2$ ($\lambda=\varepsilon$, $\theta=0$, $\alpha=1$, $a_1=a_2=a_0$):
\begin{align*}
&\forall q\in\Z,\quad |S_q(1,2,P)|=|(P*P)(q)|\le a_0^{3/4}c_{1}\cdot \varepsilon^{5/4},\\
&\forall q\in\Z^*,\quad |S_q(1,2,P)|\le 3\max\{a_0\varepsilon,c_{2}(1)\varepsilon^{1/8}\}\cdot\frac{1}{|q|^{5/4}}.
\end{align*}
If $\varepsilon$ is small enough, both $(a_0^{3/4}c_1\varepsilon^{1/8})$ and $3\max\{a_0\varepsilon,c_{2}(1)\varepsilon^{1/8}\}$ are lower than one, so that the previous relations become:
\begin{align}\label{estimS}
&\forall q\in\Z,\quad |S_q(1,2,P)|=|(P*P)(q)|\le  \varepsilon^{1+1/8},\\\label{estimSSS}
&\forall q\in\Z^*,\quad |S_q(1,2,P)|\le \frac{1}{|q|^{1+1/4}}.
\end{align}
Let us remind that the functions $S(1,i,P)$, $i\ge3$, can be written for all $q\in\Z$
\begin{align}\label{Srecur}
S(1,i,P)=P*S\left(1,i-1,P\right).
\end{align}
Using lemma~\ref{lemmeS1} with this formula and~\eqref{estimS},~\eqref{estimSSS}, this yields the following bounds for $S(1,3,P)$:
\begin{align*}
&\forall q\in\Z,\quad |S_q(1,3,P)|\le c_{1}\cdot \varepsilon^{1+1/4+1/8}= \varepsilon^{1+\theta_3} c_1\varepsilon^{1/8},\\
&\forall q\in\Z^*,\quad|S_q(1,3,P)|\le 3\max\{\varepsilon,c_{2}(\alpha_3)\varepsilon^{\alpha_3/8}\}\cdot\frac{1}{|q|^{1+\alpha_3/4}} ,
\end{align*}
where $\theta_3=\frac{1}{4}$ and $\alpha_3=\frac{1}{4}$. Let $m_a=\max\{1,a_0,a_0^{3/4}c_1,c_1\}$. Let us choose $\varepsilon$ such that 
$$3\max\{c_{2}(\alpha_3)\varepsilon^{\alpha_3/8},m_a\varepsilon^{1/8}\}\le 1:$$
since $c_2$ is a continuous decreasing function from lemma~\ref{lemmeS1}, $3\max\{a_0\varepsilon,c_{2}(1)\varepsilon^{1/8}\}\le 1$ and $a_0^{3/4}c_1\varepsilon^{1/8}\le 1$. Then the previous equations become
\begin{align*}
&\forall q\in\Z,\quad |S_q(1,3,P)|\le \varepsilon^{1+\theta_3},\\
&\forall q\in\Z^*,\quad|S_q(1,3,P)|\le \frac{1}{|q|^{1+\alpha_3/4}}.
\end{align*}
Let the two sequences $\{\theta_n\}$, $\{\alpha_n\}$ be defined by ($n\ge1$):
\begin{align*}
&\theta_1=0,\qquad \theta_n=\frac{1}{8}+\theta_{n-1},\qquad\theta_n=\frac{1}{8}(n-1),\\
&\alpha_1=1,\qquad \alpha_n=\frac{\alpha_{n-1}}{4},\qquad \alpha_n=\frac{1}{4^{n-1}}.
\end{align*}
For $N\ge 2$, if the quantity $c_2(\alpha_N)\varepsilon^{\alpha_N/8}$ is smaller than one, then $c_2(\alpha_n)\varepsilon^{\alpha_n/8}\le 1$ for all $n\le N$. Hence, assuming $3\max\{c_{2}(\alpha_N)\varepsilon^{\alpha_N/8},m_a\varepsilon^{1/8}\}\le 1$ gives by induction that each quantity $S(1,n,P)$, $n\le N$, is bounded by:
\begin{align}\label{estimSn}
&\forall q\in\Z,\quad |S_q(1,n,P)|\le \varepsilon^{1+\theta_n},\\\label{estimSnq}
&\forall q\in\Z^*,\quad|S_q(1,n,P)|\le \frac{1}{|q|^{1+\alpha_n}}.
\end{align}
\end{proof}
One can give now a first estimate on the harmonics $P_q$ of the oscillating solution of~\eqref{eqfreq}:
\begin{Lemma}\label{lemmeosc}
Suppose \Hy1, \Hy3 hold. For $\varepsilon$ small enough, relation~\eqref{eqfreq} implies the following properties:
\begin{itemize}
\item[i)] if hypothesis $1)$ of \Hy2 holds and the operator $J_{\gamma}$ is invertible for all $\gamma\in\Omega_1$, $J_{\gamma}$ is not invertible at $\gamma_0$.
\item[ii)] if \Hy2 holds, then for $e\in\Z^*$ defined in this hypothesis, $P_e\neq0$ and $\varepsilon$ can be chosen to be $\varepsilon=|P_e|$. Moreover, there exists a constant $c_0>0$ such that, for all $q\in\Z$, $q\neq e$, $|P_{q}|\le c_0 \varepsilon^{2}$.
\item[iii)] if \Hy2 holds, $P_{\beta}=0$ for all $\beta\in\Z$ such that $\beta\notin\left\{me\right\}_{m\in\Z}$.
\end{itemize}
\end{Lemma}
\begin{proof}
Suppose $J_{\gamma}$ has a bounded inverse for all $\gamma\in\Omega$, then the complex-valued function associated to the multiplication operator $\widehat{J_{\gamma}}^{-1}$ is bounded. Here this function is also denoted by the symbol $\widehat{J_{\gamma}}^{-1}$.
Since \Hy3 holds, the functions $\{g_n\}$ are bounded. Let us write $\widetilde{g}_n$ for the supremum of $|g_n|$.  Equation~\eqref{eqfreq} gives for all $q\in\Z$, $M\ge2$:
\begin{align}
|P_q|\le&\sup_k|\widehat{J_{\gamma}}^{-1}(k)|\times\sup_{n\in[2,M]} \widetilde{g}_n\times |\sum_{n=2}^{M} S_q(1,n,P)|+\nonumber\\
&+\sup_k|\widehat{J_{\gamma}}^{-1}(k)|\times|\sum_{n=M}^{\infty} S_q(g_n,n,P)|.
\end{align}
From \Hy3 and lemma~\ref{lemmetf}, the second term of the previous expression tends to zero as $M$ tends to infinity. Let us choose $M$ large enough such that 
$$|\sum_{n=M}^{\infty} S_q(g_n,n,P)|\le \sup_{n\in[2,M]} \widetilde{g}_n\times |\sum_{n=2}^{M} S_q(1,n,P)|,
$$
then there exists a constant $A=2\sup_k|\widehat{J_{\gamma}}^{-1}(k)|\times\sup_{n\in[2,M]} \widetilde{g}_n >0$ such that 
\begin{align}\label{estimJP}
|P_q|\le& A \times|\sum_{n=2}^{M} S_q(1,n,P)|.
\end{align}
Since \Hy1 holds, the hypotheses of lemma~\ref{lemmeSn1} are satisfied and for $\varepsilon$ small enough:
\begin{align}\label{1boundS}
\forall n\in[2,M],q\in\Z, |S_q(1,n,P)|\le \varepsilon^{1+\theta_n},\qquad\theta_n=\frac{1}{8}\left(n-1\right).
\end{align}
Thus for all $q\in\Z$:
\begin{align}
|P_q|\le&A \varepsilon^{1+1/8}\sum_{n=2}^{M} \varepsilon^{\theta_n-1/8}.
\end{align}
From \Hy1, one can find $\varepsilon$ small enough such that $A\varepsilon^{1/16}\sum_{n=2}^{M} \varepsilon^{\theta_n-1/8}\le1$. In consequence:
\begin{align}\label{boundP}
|P_q|\le\varepsilon^{1+\sigma}\qquad \sigma=1/16.
\end{align}
For $i)$, 
using this new bound as hypothesis for lemma~\ref{lemmeSn1} (replacing $\varepsilon$ by $\varepsilon^{1+\sigma}$), gives for all $q\in\Z$ a new, stronger, bound on the $S(1,n,P)$ and thus on $P_q$, i.e.:
\begin{align}\label{boundP2}
|P_q|\le&\varepsilon^{(1+\sigma)^2}.
\end{align}
Since $\sigma>0$, doing this reasoning again, an infinite number of times, will lead to $P_q=0$, for all $q$.

Proof of $ii)$. Here $J_{\gamma}$ is not invertible for all $\gamma\in\Omega$ but $(QJ_{\gamma}Q)^{-1}$ is since \Hy2 holds. Hence, estimates~\eqref{estimJP} and~\eqref{boundP} hold for all $q\neq \pm e$ and $\varepsilon$ small enough ($e$ being the integer defined in \Hy2). Let $s=\{e,-e,q-e,q+e\}$. The function $S(1,2,P)$ at $q\in\Z$ can be written as:
\begin{align}\label{devS}
S_q(1,2,P)=&2P_{q-e}P_{e}+2P_{q+e}P_{-e}+\sum_{\stackrel{i\in\Z}{i\notin s}} P_{q-i}P_i.
\end{align}
The first two terms are bounded by a quantity proportional to $\varepsilon^{2+\sigma}$ in the case where $q\neq\pm2e$ and $q\neq0$ and $\varepsilon^{2}$ otherwise. The sum over $i$ contains only terms bounded by $\varepsilon^{1+\sigma}$ so has a bound proportional to $\varepsilon$ to the power $(1+\sigma)(1+1/8)$, consequence of lemma~\ref{lemmeSn1}.
Similarly, for $q\in\Z$ and $n\in[3,M]$: 
\begin{align}\label{devS2}
S_q(1,n,P)=(P*S(1,n-1,P))(q)=&P_{e}S_{q-e}(1,n-1,P)+P_{-e}S_{q+e}(1,n-1,P)+\nonumber\\
&+P_{q-e}S_{e}(1,n-1,P)+P_{q+e}S_{-e}(1,n-1,P)+\nonumber\\
&+\sum_{\stackrel{i\in\Z}{i\notin s}} P_{q-i}S_i(1,n-1,P).
\end{align}
Similarly, assuming $|S_q(1,n-1,P)|\le\varepsilon^{1+\sigma}$ for all $q\neq\pm e$ and using lemma~\ref{lemmeSn1}, for $\varepsilon$ small enough, yields: 
\begin{align}
|\sum_{\stackrel{i\in\Z}{i\notin s}} P_{q-i}S_i(1,n-1,P)|\le \varepsilon^{(1+\sigma)(1+\theta_n)},
\end{align}
hence
\begin{align}\label{boundSall}
|S_q(1,n,P)|\le b \varepsilon^{(1+\sigma)(1+\theta_n)},\qquad b>0.
\end{align}
Since the quantity $QJQ$ is invertible, for all $q\neq \pm e$ the estimate~\eqref{estimJP} holds true. Let $\lambda=\varepsilon^{1+\sigma}$. Processing like in $i)$:
\begin{align}
|P_q|\le&A b \lambda^{1+1/8}\sum_{n=2}^{M} \lambda^{\theta_n-1/8},
\end{align}
and for $\varepsilon$ small enough such that $|A b \lambda^{1/16}\sum_{n=2}^{M} \lambda^{\theta_n-1/8}|\le 1$
\begin{align}\label{boundPq}
|P_q|\le&\lambda^{1+\sigma} \le \varepsilon^{(1+\sigma)^2}.
\end{align}
This procedure can be repeated several times and gives each time a larger power of $\varepsilon$, until the first two terms of~\eqref{devS} and four terms of~\eqref{devS2} become the largest terms. This implies that for all $q\in\Z$, $q\neq \pm e$, $|P_q|\le|P_e|$. If $P_e=0$, then no oscillation exists. If $P_e\neq0$, the largest term $P_e$ can be chosen to be $|P_e|=\varepsilon$ and there exists $c_0>0$ such that for all $q\in\Z$, $q\neq \pm e$
\begin{align}
|P_q|\le& c_0 \varepsilon^{2}.
\end{align}
For $q\neq\pm 2e$ the bound is even smaller as it can be seen from the leading terms of~\eqref{devS} and~\eqref{devS2}.

For $iii)$,
let us define the set $s_e=\left\{me\right\}_{m\in\Z^*}$. Let $\beta\in\Z$. If $\beta\notin s_e$, then for each $n\in\Z$ fixed, at least one of the quantities $(\beta-n)$ or $n$ does not belong to $s_e$. For all $\beta\notin s_e$, $ii)$ gives
\begin{align}\label{step1P}
|P_\beta|\le\varepsilon^{2}.
\end{align}
The first two terms of~\eqref{devS} are bounded by $\varepsilon^{3}$ and from lemma~\ref{lemmeS1}, one can write:
$$|\sum_{\stackrel{i\in\Z}{i\notin s}}P_{\beta-i}P_i|\le \varepsilon^{\theta}, \quad \theta=2+\frac{1}{4}.
$$
Following the procedure used in the proof of $ii)$, one will find that $\theta$ is increasing at each step. Eventually, $\theta$ tends to infinity and implies for $\varepsilon$ small enough:
$$P_{\beta}=0.
$$
\end{proof}
\begin{proof}[Proof of Theorem~\ref{Thbif}]
The proof is a direct application of the results given in lemma~\ref{lemmeosc}. Notice first that in order to have a solution satisfying \Hy1, the point $i)$ of this lemma requires that $J_{\gamma}$ possesses at least one zero eigenvalue at $\gamma_0\in\Omega$, showing the necessity of a hypothesis such as \Hy2. As a consequence, the bifurcation point is the point for which an eigenvalue of $J_{\gamma}$ reaches zero. The point $ii)$ of the same lemma states that the Fourier series of the oscillating solution must include a non zero value for $P_e$, where $e$ is given in \Hy2. The point $iii)$ implies that $P_e$ is the term of the Fourier series associated to the frequency of oscillation, since only harmonics associated with frequencies multiple of $e$ can be different from zero. Eventually, the bound on the Fourier components given in $ii)$ of lemma~\ref{lemmeosc} finishes the proof.
\end{proof}

\section{Decrease of the harmonics}\label{decrease}
This section concerns the estimation of the Fourier components amplitude of the solution. It is dedicated to theorem~\ref{Th1} and its proof.

Before stating the result, notice that the index $e$ of \Hy2 can be chosen to be $e=1$ without loss of generality. Sections $ii)$ and $iii)$ of lemma~\ref{lemmeosc} assert that only $P_q$ with $|q|\in\{e,2e,3e,...\}$ are not null, so that $e$ determines the oscillating frequency: $\omega_e$.
\begin{Theorem}\label{Th1}
Under the hypothesis of Theorem~\ref{Thbif}, for all $\gamma\in\Omega$ where $|P_1|$ is small enough, there exists $k=k(|P_1|)\ge2$ such that
the oscillating solution of the system~\eqref{eqfreq} possesses the following properties:
for all $q\in\Z^*$,
\begin{align}\label{estimPk}
&\ |P_{q}|\le c_{q}|P_1|^{|q|},\quad \text{for all } |q|\le k,\nonumber\\
&\ |P_{q}|\le c_{k}|P_1|^k,\quad \text{for all } |q|\ge k,
\end{align}
where $c_q,c_k$ are positive constants.
Moreover, the rank $k$ increases when $|P_1|$ decreases and tends to infinity when $|P_1|$ tends to zero.
\end{Theorem}
\begin{Remark}
If some of the transfer functions $\{\widehat{Z_{i,j}}\}$ change sign or decrease fast enough when their variable tends to infinity, the decrease of the harmonics amplitude with frequency may be faster than the estimation given in the above theorem.
\end{Remark}
\begin{proof}[Proof of theorem~\ref{Th1}]
The reasoning of the proof is similar to the one used in the proof of lemma~\ref{lemmeosc}. 
The operator $Q\widehat{J_{\gamma}}Q$ has a bounded inverse for all $\gamma\in\Omega$ since \Hy2 holds and the functions $\{g_n\}$ are bounded since \Hy3 holds. Let us write $\widetilde{g}_n$ for the supremum of $|g_n|$.  Equation~\eqref{eqfreq} gives for all $q\in\Z\backslash\{1,-1\}$, $M\ge2$:
\begin{align}
|P_q|\le&\sup_k|(Q\widehat{J_{\gamma}}Q)^{-1}(k)|\times\sup_{n\in[2,M]} \widetilde{g}_n\times |\sum_{n=2}^{M} S_q(1,n,P)|+\nonumber\\
&+\sup_k|(Q\widehat{J_{\gamma}}Q)^{-1}(k)|\times|\sum_{n=M}^{\infty} S_q(g_n,n,P)|.
\end{align}
From \Hy3 and lemma~\ref{lemmetf}, the second term of the previous expression tends to zero as $M$ tends to infinity. Let us choose $M$ large enough such that 
$$|\sum_{n=M}^{\infty} S_q(g_n,n,P)|\le \sup_{n\in[2,M]} \widetilde{g}_n\times |\sum_{n=2}^{M} S_q(1,n,P)|,
$$
then there exists a constant $A>0$ such that 
\begin{align}\label{estimJP2}
|P_q|\le& A \times|\sum_{n=2}^{M} S_q(1,n,P)|\qquad\forall q\neq\pm1.
\end{align}
Let $s$ be the set $s=\left\{0,1,-1,2,-2,q-1,q+1,q-2,q+2\right\}$.
The function $S(1,2,P)$ at $q\in\Z$ can be written as:
\begin{align}\label{devSp}
S_q(1,2,P)=&2P_{q-1}P_{1}+2P_{q+1}P_{-1}+2P_{q-2}P_{2}+2P_{q+2}P_{-2}+\sum_{\stackrel{i\in\Z}{i\notin s}} P_{q-i}P_i.
\end{align}
From Th.~\ref{Thbif}, for all $|q|\ge3$ the first four terms of~\eqref{devSp} are bounded by a quantity proportional to $\varepsilon^{3}$. In addition, each term of the sum over $i$ contains two quantities, each one is bounded by $c_0\varepsilon^{2}$ for some $c_0>0$, then lemma~\ref{lemmeS1} (with $\lambda=(\sqrt{c_0}\varepsilon)^2$) gives the estimate:
\begin{align}
\sum_{\stackrel{n\in\Z}{n\notin s}} |P_{q-n}P_n|\le (\sqrt{c_0}\varepsilon)^{2(1+1/4)}=c_0^{2+1/2}\varepsilon^{1/4}\varepsilon^{2+1/4}.
\end{align}
As a consequence, for $|q|\ge3$, and for $\varepsilon$ small enough, $|S_q(1,2,P)|\le \varepsilon^{2+\theta}$ where  $\theta=1/4$. Let us note for the following that in the case where $q=0$, $q=\pm1$ and $q=\pm2$, some of the first terms of~\eqref{devSp} are bounded by $\varepsilon^2$, so for $q\in\{0,\pm1,\pm2\}$, there exists $c_{2,2}>0$ such that $|S_q(1,2,P)|\le c_{2,2}\varepsilon^{2}$. Let us now give a bound on $S(1,n,P)$ by induction. Assume that at rank $n-1\ge2$ the following statements are true:
\begin{itemize}
 \item[a.1)] for $q\in\{0,\pm1,\pm2\}$, $|S_q(1,n-1,P)|\le C_{2}\varepsilon^{2+\theta_{n-1}}$,
\item[b.1)] for $q$, $|q|\ge3$, $|S_q(1,n-1,P)|\le \varepsilon^{2+\theta_n}$,
\end{itemize}
for some $C_{2}>0$ and $\theta_n=1/4(n-1)$. Then at rank $n$, for $q\in\Z$ : 
\begin{align}\label{devSS}
S_q(1,n,P)=(P*S(1,n-1,P))(q)=&U_q^{n}+\sum_{\stackrel{i\in\Z}{i\notin s}} P_{q-i}S_i(1,n-1,P),
\end{align}
where, for $s_2=\{0,\pm1,\pm2\}$,
$$U_q^{n}=\sum_{k\in s_2}P_{k}S_{q-k}(1,n-1,P)+P_{q-k}S_{k}(1,n-1,P).$$
Firstly, for $|q|\le 2$
\begin{align*}
|U_q^{n}|&\le5 c_0C_2\varepsilon^{2+\theta_{n-1}+1}=5c_0C_2\varepsilon^{3/4}\cdot\varepsilon^{2+\theta_{n-1}+1/4}\\
&\le\varepsilon^{2+\theta_{n}},
\end{align*}
for $\varepsilon$ such that $5c_0C_2\varepsilon^{3/4}\le1$.
For $|q|\ge 3$,
$$|U_q^{n}|\le\varepsilon^{2+\theta_{n}+1/4}.
$$
Secondly, for all $q\in\Z$ the sum over $i$ of~\eqref{devSS} is bounded, from lemma~\ref{lemmeS1} and for $\varepsilon$ small enough, by:
\begin{align*}
\sum_{\stackrel{i\in\Z}{i\notin s}}\left| P_{q-i}S_i(1,n-1,P)\right|\le\lambda^{1+1/4+\theta_{n-1}/2} a_1^{3/4}C_1c_0\le \lambda^{1+1/8+\theta_{n-1}/2},\ \text{ where } \lambda=\varepsilon^{2}.
\end{align*}
This proves relations $a)$ and $b)$ by induction.
In consequence~\eqref{estimJP2} implies that there exists a constant $D>0$, such that for all $|q|\ge3$ (see proof of lemma~\ref{lemmeosc} for details):
\begin{align}\label{boundPq3}
|P_q|\le& \varepsilon^{2+\theta},\qquad \sigma=1/8.
\end{align}
Similarly to the proof of $ii)$ of lemma~\ref{lemmeosc}, this procedure can be repeated several times and gives each time a larger power of $\varepsilon$, until the first terms of~\eqref{devSp} become the largest terms. In conclusion, there exist $D>0$ such that one has:
\begin{align}
\text{for } q\ge3,\quad|P_q|\le& D \varepsilon^{3},\qquad D=9.
\end{align}

For the general case the same reasoning is used but the constant $D$ depends on $q$ and increases with it. Suppose there exists a rank $r\ge3$ and a function $D:\Z\to\R_+$ such that :
\begin{itemize}
 \item[a.2)] for all $q$, $|q|\le r$, $|P_q|\le D(q)\varepsilon^{q}$,
\item[b.2)] for all $q$, $|q|\ge r$, $|P_q|\le D(r)\varepsilon^{r}$.
\end{itemize}
Let us show it is true at rank $r+1$ under appropriate conditions. The steps are the same as for the case $r=3$ treated previously. In the following, only the case $q$ positive is treated as the case $q$ negative will lead to the same estimates.
For $q\in[2,r]$, let us introduce the set $s_q=\Z\cap[1,q-1]$. Let $D$ be the function satisfying:
\begin{equation}
\forall q\ge2,\quad D(q)=2A\sum_{i\in s_q}D(q-i)D(i).
\end{equation}
Notice that the function $D$ is monotone increasing on $\N$.
For the estimation of $P_{r+1}$, let us introduce for $n\ge 3$
$$U_{r+1}^2=\sum_{i\in s_{r+1}}P_{r+1-i}P_i,\quad U_{r+1}^n=\sum_{i\in s_{r+1}}P_{r+1-i}S_i(1,n-1,P).
$$
Instead of~\eqref{devSp}, one has to estimate the quantity:
\begin{equation}\label{Usumi}
S_{r+1}(1,2,P)=U_{r+1}^2+\sum_{i\notin s_{r+1}}P_{r+1-i}P_i.
\end{equation}
The quantity $U_{r+1}^2$ is bounded by:
$$|U_{r+1}^2|\le \varepsilon^{r}\sum_{i\in s_{r+1}}D(r+1-i)D(i)=\frac{D(r+1)}{2A}\varepsilon^{r}.
$$
Let us choose $\varepsilon$ such that $\varepsilon^{1/16}D(r)^2\le D(r)/(2MA)$. Let us notice that this relationship implies a dependence between $\varepsilon$ and $r$.
The sum over $i$ in~\eqref{Usumi} is bounded by $D(r)^2\varepsilon^{r+1/8}$ from lemma~\ref{lemmeSn1}. Hence, for $\varepsilon$ small enough:
\begin{equation}\label{qinfr}
|S_{r+1}(1,2,P)|\le \frac{D(r)}{2MA}\varepsilon^{r+1/16}.
\end{equation}
For $q> r+1$, let us introduce the set $t_r=\Z\cap [-r,r]$ and for $n\ge 3$
$$V_q^2=\sum_{i\in t_{r+1}}P_{q-i}P_i,\quad V_q^n=\sum_{i\in t_{r+1}}P_{q-i}S_i(1,n-1,P).
$$
Instead of~\eqref{devSp}, one has to estimate the quantity:
\begin{equation}\label{SVsum}
S_q(1,2,P)=V_q^2+\sum_{i\notin t_{r+1}}P_{q-i}P_i.
\end{equation}
The quantity $V_q^2$ contains terms bounded by $D(r)^2\varepsilon^{r+1}$ or smaller bounds. So for $\varepsilon$ such that $\varepsilon^{\frac{15}{16}}D(r)^2 \le D(r)/(4MAr)$ one obtains:
$$|V_q^2|\le 2rD(r)^2\varepsilon^{r+1}\le \frac{D(r)}{2MA}\varepsilon^{r+1/16}.
$$
For the sum over $i$ of~\eqref{SVsum}, each term contains two quantities where at least one is bounded by $D(r)\varepsilon^r$. Processing like in the proof of lemma~\ref{lemmeSn1}:
\begin{equation*}
|\sum_{i\notin t_q}P_{q-i}P_i|\le D(r)^2\varepsilon^{r+1/8}\le \frac{D(r)}{2MA}\varepsilon^{r+1/16},
\end{equation*}
for $\varepsilon$ such that $\varepsilon^{1/16}D(r)^2\le D(r)/(2MA)$.

Next, one has to treat the case of multiple convolutions. Similarly to the case $r=3$, suppose at rank $n-1\in[2,M]$ there exist a function $h_{n-1}:\Z\to\R_+$ such that:
\begin{itemize}
 \item[a.3)] for $q\in [-r,r]$, $|S_q(1,n-1,P)|\le h_{n-1}(q)\varepsilon^{q}$,
\item[b.3)] for $q$, $|q|> r$, $|S_q(1,n-1,P)|\le h_{n-1}(r)\varepsilon^{r+1/16}$,
\end{itemize}
and where $h$ obeys
$$h_{n}(q)=2MA\varepsilon\sum_{i\in s_q}f(q-i)h_{n-1}(i).$$
At rank $n$, for $|q|\le r$ : 
\begin{align}\label{devSSg}
S_{q}(1,n,P)=(P*S(1,n-1,P))(q)=&U_{q}^{n}+\sum_{\stackrel{i\in\Z}{i\notin s_{q}}} P_{q-i}S_i(1,n-1,P),
\end{align}
and for $|q|> r+1$ :
\begin{align}\label{devSSg2}
S_q(1,n,P)=&V_q^{n}+\sum_{\stackrel{i\in\Z}{i\notin t_{r+1}}} P_{q-i}S_i(1,n-1,P).
\end{align}
It is similar to the previous case, where $n=2$. Firstly, the quantity $U_{q}^{n}$ is bounded by
$$|U_{q}^{n}|\le \frac{h_{n}(q)}{2MA}\varepsilon^{q}.
$$
Secondly, the term $V_q^{n}$ is bounded by $h_{n-1}(r)^2 r\varepsilon^{r+1/16+1}$ and the sum over $i$ of~\eqref{devSSg} (resp.~\eqref{devSSg2}) are bounded by $h_{n-1}(r)^2\varepsilon^{q+1/8}$ (resp. $h_{n-1}(r)^2\varepsilon^{r+1/8}$).
Let us assume that $\varepsilon$ is such that for all $n\in[3,M+1]$:
$$h_{n-1}(r)^2 2 r \varepsilon^{15/16}\le \frac{h_n(r)}{2AM}\ \text{ and }\ h_{n-1}(r)^2\varepsilon^{1/16}\le \frac{h_n(r)}{2AM}.
$$
In consequence, $a.3)$ and $b.3)$ are true at rank $n$ and this finishes the estimates on the multiple convolutions. 

The last step of the proof is to use~\eqref{estimJP2} which leads to:
$$\forall q\ge r+1,\ |P_q|\le D(r+1)\varepsilon^{r+1/16}.
$$
Replacing $b.2)$ with this estimate and doing each step of the proof again will gives
$$\forall q\ge r+1,\ |P_q|\le D(r+1)\varepsilon^{r+2/16},
$$
and doing it many times will lead to the bound $D(r+1)\varepsilon^{r+1}$.
To obtain these estimates, $\varepsilon$ must be bounded by a quantity which depends on $r$, and which decreases when $r$ increases. Hence, for $\varepsilon$ fixed, there is a $r$ maximal. Eventually, Th.~\ref{Thbif} gives $|P_1|=\varepsilon$ and this concludes the proof.
\end{proof}

\section{Bifurcation point and stability}\label{stable}

In this section, the regime of the system around the bifurcation point is investigated. The direction of the bifurcation, the amplitude of the first Fourier component and the influence of $\gamma$ on the oscillation frequency are given. The results are similar to the ones arising in the finite dimensional Hopf bifurcation. As for the Hopf theorem, an additional hypothesis on the system is required, i.e. the transversality condition stated in \Hy4.

Without loss of generality, as explained in the previous part, the index $e$ of \Hy2 is assumed to be $e=1$ in the following. 
Suppose \Hy2 holds, then for $\gamma\in\Omega$, $\beta_2(\gamma)$, the eigenvalue of $J_{\gamma}$ associated to $\omega_2$, is not null. Assume \Hy3 and let us introduce the notation:
\begin{equation}\label{defD}
D_1(\gamma)=[\beta_2(\gamma)]^{-1}g_{2}(2,1)\left(g_{2}(1,2)+g_{2}(1,-1)\right)-\left(g_{3}(1,0,1)+g_{3}(1,0,-1)+g_{3}(1,2,1)\right).
\end{equation}
Let us denote by $\Re[z]$ the real part and $\Im[z]$ the imaginary part of the complex number $z$. In this section, the following additional hypothesis is assumed:
\begin{itemize}
\item[\Hy4] There exists an open set $\Gamma$ with $\gamma_0\in\Gamma$ such that the complex-valued functions $D_1$ and $\beta_1$ are respectively continuously differentiable and twice continuously differentiable with respect to $\gamma$ on $\Gamma$. 
Moreover, $D_1(\gamma_0)\neq0$ and
$$\Re[ \frac{1}{D_1(\gamma_0)}\frac{d\beta_1}{d\gamma}(\gamma_0)]\neq0.
$$
\end{itemize}
\begin{Theorem}\label{ThP1}
Suppose \Hy1, \Hy2, \Hy3, \Hy4 hold. Then, 
\begin{itemize}
\item[i)] The quantity 
\begin{align}\label{defalpha}
\alpha=\frac{1}{D_1(\gamma_0)}\frac{d\beta_1}{d\gamma}(\gamma_0),
\end{align}
is a real number, different from zero. If its sign is positive (resp. negative) the bifurcation is direct (resp. inverse), i.e. the oscillating solution described in \Hy1 exists for $\gamma>\gamma_0$ (resp. $\gamma<\gamma_0$).
\item[ii)] for all $\gamma\in\Gamma$ such that $|\gamma-\gamma_0|$ is small enough, the modulus of the first Fourier component of the oscillating solution is given by:
\begin{align}
|P_1|=\sqrt{|\alpha|}\sqrt{|\gamma-\gamma_0|}+{\cal O}(|\gamma-\gamma_0|^{3/2}).
\end{align}
\end{itemize}
\end{Theorem}
\begin{proof}
For $\varepsilon$ small enough, from Th.~\ref{Th1} and lemma~\ref{lemmeS1}, one can write:
\begin{align*}
&S(g_2,2,1)=\left(g_{2}(1,2)+g_{2}(1,-1)\right)P_{-1}P_2+{\cal O}(\varepsilon^4),\\ &S(g_3,3,1)=\left(g_{3}(1,0,1)+g_{3}(1,0,-1)+g_{3}(1,2,1)\right)|P_{1}|^2P_1+{\cal O}(\varepsilon^4),\\
&\forall n\ge4,\quad |S(g_n,n,1)|\le c_1\varepsilon^4,\quad\text{with } c_1>0.
\end{align*}
One has also:
\begin{align*}
&S(g_2,2,2)=g_{2}(2,1)P_{-1}P_1+{\cal O}(\varepsilon^4),\\ 
&\forall n\ge3,\quad |S(g_n,n,2)|\le c_2\varepsilon^4,\\
&\forall n\ge2,q\ge3,\quad |S(g_n,n,q)|\le c_3\varepsilon^4.
\end{align*}
Then it yields for $P_2$:
\begin{align*}
J_{\gamma}P_2&=\beta_2(\gamma)P_2=\sum_{i=2}^{\infty}S(g_i,i,2)=g_{2}(2,1)P_{-1}P_1+{\cal O}(\varepsilon^4)\\
P_2&=[\beta_2(\gamma)]^{-1}g_{2}(2,1)P_{-1}P_1+{\cal O}(\varepsilon^4),
\end{align*}
For $P_1$:
\begin{align*}
J_{\gamma}P_1&=\sum_{i=2}^{\infty}S(g_i,i,1)\\
&=\left(g_{2}(1,2)+g_{2}(1,-1)\right)P_{-1}P_2+\left(g_{3}(1,0,1)+g_{3}(1,0,-1)+g_{3}(1,2,1)\right)|P_{1}|^2P_1+{\cal O}(\varepsilon^4),
\end{align*}
which yields
\begin{align*}
&P_1\left\{\beta_1(\gamma)-D_1(\gamma)|P_{1}|^2\right\}+{\cal O}(\varepsilon^4)=0,
\end{align*}
where $D_1$ is the differentiable function defined in~\eqref{defD}. For the oscillating solution, where $P_1\neq0$, one can write
\begin{align*}
|P_1|^2D_1(\gamma)=\beta_1(\gamma)+{\cal O}(\varepsilon^4),
\end{align*}
From the hypotheses of \Hy4 and \Hy2, one has around $\gamma_0$ the following Taylor expansions:
\begin{align*}
D_1(\gamma)=D_1(\gamma_0)+{\cal O}(|\gamma-\gamma_0|),\quad \beta_1(\gamma)=\beta_1'(\gamma_0)(\gamma-\gamma_0)+{\cal O}(\gamma-\gamma_0)^{2}.
\end{align*}
and
\begin{align}\label{eqP1}
|P_1|^2=\frac{\beta_1'(\gamma_0)}{D_1(\gamma_0)}(\gamma-\gamma_0)+{\cal O}(\gamma-\gamma_0)^2+{\cal O}(\varepsilon^4).
\end{align}
For $\varepsilon$ small enough, one has $|P_1|=\varepsilon$ from Th~\ref{Thbif} and since the numerator is not null from \Hy4, equation~\eqref{eqP1} shows that $(\gamma-\gamma_0)$ is proportional to $\varepsilon^2$. Moreover, the left hand side of~\eqref{eqP1} is a positive quantity, implying:
\begin{align}
\left\{\begin{array}{cc}
&\Re[\alpha](\gamma-\gamma_0)>0\\
&\Im[\alpha]=0
\end{array}\right..
\end{align}
Hence the sign of $\Re [\alpha]$ determines the direction of the bifurcation (direct or inverse).
\end{proof}

The last result of this section concerns the influence of $\gamma$ on the oscillation frequency, $\omega_1$, around the threshold. It requires two further hypotheses, one on the linear part of the system and one on the coefficients in front of the quadratic and cubic terms, contained in the function $D_1$. The proof is in fact an application of the implicit function theorem. For the next result, $\beta_1$, $D_1$ shall be written as functions of two variables: $\beta(\omega_1,\gamma)=\beta_1(\gamma)$, $D(\omega_1,\gamma)=D_1(\gamma)$.
\begin{Corollary}\label{corolomega}
Under the hypotheses of Th.~\eqref{ThP1}, suppose the function $\Im[D\beta]$ is twice continuously differentiable with respect to its variable $\omega_1$ and $\gamma$ in an open set around the point $(\omega_1,\gamma_0)$, and 
$$\left.\frac{\partial}{\partial \omega_1}\Im[D\beta_1]\right|_{(\omega_1,\gamma_0)}\neq0,\quad \left.\frac{\partial^2}{\partial \omega_1^2}\Im[D\beta_1]\right|_{(\omega_1,\gamma_0)}\neq0.
$$ 
Suppose further that $\beta$ is differentiable with respect to both variables $\omega_1$ and $\gamma$ at $(\omega_1,\gamma_0)$. Then the expression of the angular frequency $\omega_1$ for $\gamma$ in a neighbourhood of $\gamma_0$ is given by:
\begin{equation}\label{omegaT}
\omega_1(\gamma)=\omega_1(\gamma_0)+\omega_1'(\gamma_0)\cdot (\gamma-\gamma_0)+{\cal O}(\gamma-\gamma_0)^2,
\end{equation}
where $\omega_1(\gamma_0)$ is the solution of
$\Im [\beta(\omega_1(\gamma_0),\gamma_0)]=0$
and
$$\omega_1'(\gamma_0)=-\frac{\Im\left[D\frac{\partial \beta}{\partial \gamma} \right]{(\omega_1(\gamma_0),\gamma_0)}}{\Im\left[D\frac{\partial \beta}{\partial \omega} \right]{(\omega_1(\gamma_0),\gamma_0)}}.
$$
\end{Corollary}
\begin{proof}
The Taylor expansion in Eq.~\eqref{omegaT} is a consequence of the implicit function theorem applied to the twice differentiable function $\Im[D\beta]$ around $\gamma_0$.
The value of $\omega_1(\gamma_0)$ is given by Th.~\ref{Thbif}. The hypothesis \Hy2 $2)$ and the point $i)$ of
Th.~\ref{ThP1}, imply that
\begin{equation}\label{condI0}
\Im[\frac{dD_1\beta_1}{d\gamma}](\gamma_0)=\Im\left[D_1(\gamma_0)\frac{d \beta_1}{d \gamma}(\gamma_0)\right]=0.
\end{equation}
The implicit function theorem and the hypotheses of differentiability allow to write
\begin{align*}
& \Im\left[D \frac{d\beta_1}{d\gamma} \right](\gamma_0)=\frac{d\omega}{d\gamma}(\gamma_0)\cdot\Im\left[D\frac{\partial \beta}{\partial \omega} \right]{(\omega_1(\gamma_0),\gamma_0)}+\Im\left[D\frac{\partial \beta}{\partial \gamma} \right]{(\omega_1(\gamma_0),\gamma_0)}.
\end{align*}
Putting this latter expression in Eq.~\eqref{condI0} concludes the proof.
\end{proof}

\section{Example of the woodwind musical instruments}\label{woodwind}
A reed woodwind musical instrument, such as a clarinet or a saxophone, is an example of a system which can be described by an hyperbolic PDE and a localized polynomial nonlinearity. 
In acoustics, the major challenge is to understand the influence on the produced sound of physical parameters such as the pressure in the mouth of the musician, the stiffness of the reed, the shape of the mouthpiece and of the resonator. Assuming it can be done, a first order calculation by linearizing the system gives the bifurcation point and the frequency of oscillation at this point. But the investigation of these systems requires to go further and to obtain the shape of the Fourier harmonics with respect to the bifurcation parameter. That is why the frequency domain approach is perfectly appropriate here.
For small oscillations, one can obtain analytical formulae giving information on the solution of the system, but, as said in the introduction this requires until now a hypothesis on the Fourier series of the solution. Let $p_0\in L^2(T)$ be the acoustic pressure inside the mouthpiece of the wind instrument under study and let us suppose it is oscillating. The hypothesis made by acousticians, see e.g.~\cite{GGL}, \cite{KOG}, \cite{SKVG}, \cite{RGKSV}, is that the $q$-th harmonic of this pressure obeys:
\begin{align}\label{eqWorman}
|P_q|\le|P_1|^{|q|},\quad q\neq0.
\end{align}
This hypothesis is now going to be proved in the framework of the mathematical results obtained in the previous sections.

The equations describing this system are detailed in e.g.~\cite{FR}, \cite{KC}. Although the instrument possesses a three dimensional shape, the model is one dimensional and $X=[0,\ell]$ where $\ell$ is the length of the instrument. The dimensionless variables associated to the system are the acoustic pressure, $p$, and the volume flow, $u$, inside the instrument, both belonging to ${\cal X}$, and $h\in L^2(T)$ the reed tip opening localized $x=0$. The pressure inside the mouth of the instrumentalist $\gamma$ is the bifurcation parameter. Let us call $p_0$ and $u_0$ the pressure and volume flow at $x=0$. The acoustic pressure depends on the volume flow entering the mouthpiece through the relation:
\begin{align}
Ap=\delta(x)Bu,
\end{align}
where $B=\partial/\partial t$ and $A$ is the differential operator associated to the wave equation inside the instrument, with losses at the boundaries and along the resonator (visco-thermal losses). By using the Green functions formalism, it can be written as
\begin{align}
p(t,x)=\int_{T\times X}g(t-\tau,x-y)Bu(\tau,y)\delta(y)dyd\tau=\int_{T}g(t-\tau,x)Bu(\tau,0)d\tau,
\end{align}
where $g$ is the Green function solution of $Ap=\delta(x)\delta(t)$. An example of this Green function for a model of clarinet is given in~\cite{KC}, Chapter II.5, section 5.2: the model considered is the wave equation with visco-thermal losses on $X$, Neumann boundary condition at one end and Dirichlet boundary condition at the other end. This gives for the time Fourier transform $\widehat{g}$ of $g$:
\begin{align}
\widehat{g}(x,\omega;x_s)=c\sum_{n>0}\frac{\cos(k_nx)\cos(k_nx_s)}{\omega_n^2+i\omega\omega_nq_n-\omega^2},
\end{align}
where $c$ is a constant, $k_n=(2n-1)\pi/(2\ell)$ and $\omega_n$ is proportional to $k_n$. The term $q_n$ is proportional to $\sqrt{k_n}$ in order to model visco-thermal losses. This choice of $q_n$ is motivated by physical arguments but leads to a fractional derivative in the expression of $A$. This is partly why the present study start with Eq.~\eqref{eqreel}, where $Z$ is much more convenient to define than $A$.
Let $U={\cal F}u_0$, the above equation expressed in the Fourier space, and at $x=0$, yields:
\begin{align}
P=\widehat{Z}U,
\end{align}
where $\widehat{Z}$ is called a transfer function.
This is the equation~\eqref{eqreel} in the case of wind instruments, in the frequency domain. Indeed, $U$ is expressed in terms of $P$ through a nonlinear relation in the following. Notice that the quantity $\widehat{Z}$ is called by acousticians the input impedance of the bore and can be measured thanks to physical experiments. It is bounded, provided losses are taken into account in the model. See~\cite{KC}, chapter II.7, for examples of impedances of various wind instruments. The Fourier basis is the eigenbasis of this operator, thus in the Fourier space, $\widehat{Z}$ is a multiplication operator. For all $q\in\Z$,
\begin{align}\label{PZU2}
P_q=\widehat{Z}(\omega_q)U_q,
\end{align}
where the complex-valued function $\widehat{Z}(\cdot)$ is never null except maybe at $\omega_0=0$.

The nonlinear relationship is localized at the tip of the mouthpiece and relates $u_0$, $p_0$ and $h$. Under some physical hypotheses, the Bernoulli law at the entry of the instrument gives the nonlinear relationship at time $t\in\R$:
\begin{align}\label{eqnlt2}
u_0^2(t)=\zeta^2\left(1-\gamma+h(t)\right)^2\left(\gamma-p_0(t)\right),
\end{align}
where $\zeta$ is a positive constant, $0<\gamma<1$. At rest, when there is no action from the instrumentalist, $\gamma=0$, and the equilibrium position of the instrument yields $u_0=0$, $p_0=0$ and $h=0$. It is supposed that for all $t$, $p_0(t)$ and $h(t)$ are small enough such that $\gamma-p_0(t)\ge 0$ and $1-\gamma+h(t)\ge 0$ (small oscillations). 

Let $H$ denote the Fourier transform of $h$. The reed opening is related to $p_0$ via a linear relation expressed in the frequency domain:
\begin{align}\label{EDP}
H_q&=\widehat{D}(q)P_q,\qquad\text{for all}\ q\in\Z.
\end{align}
The complex-valued function $\widehat{D}$ is bounded. Let us introduce $Q_0$ the orthogonal projection associated to the constant part of the Fourier decomposition. From~\eqref{PZU2}, one has
\begin{align}\label{PZU}
U_q&=\widehat{Y}(q)(1-Q_0)P_q+U_0Q_0,\quad\text{for all } q\in\Z.
\end{align}
In the physical system under study, the function $\widehat{Y}(1-Q_0):\Z\to\C$ is bounded. Nevertheless, the value $\widehat{Y}(0)$ is not defined as it is assumed that $P_0=0$ and $U_0>0$. It gives a good opportunity to show how to cope with this problem and satisfy the hypothesis of boundedness of the operators despite it: in this case the system will be written as a set of two equations, with a specific one for $U_0$, see below how to obtain~\eqref{eqNLP2} and~\eqref{eqNLP20}.
Behind the function $\widehat{Y}$ is hidden the hyperbolicity given by the wave equation inside the resonator and the contribution of $B$. It reveals itself when looking at the limit of $|\widehat{Y}(q)|$ when $q$ tends to infinity: it tends to a strictly positive constant.

Let us assume the pressure $p_0$ satisfies hypothesis \Hy1.
Let us recall that since $\widehat{D}$ is bounded, one has for example, for all $q\in\Z$,
\begin{align*}
&\widehat{hp}(q)=(H*P)(q)=\left((\widehat{D}P)*P\right)(q)=\sum_{n\in\Z}\widehat{D}_{q-n}P_{q-n}P_n,
\end{align*}
as shown by lemma~\ref{lemmetf}.
If, like in~\cite{RGKSV}, the following notations are introduced:
\begin{align}
&u_{00}=\zeta^2\gamma(1-\gamma)^2,\quad A_q=2\zeta^2\gamma(1-\gamma)\widehat{D}(q)-\zeta^2(1-\gamma)^2,\nonumber\\
&\textbf{B}_{q,n}=\zeta^2\gamma \widehat{D}(q-n)\widehat{D}(n)-2\zeta^2(1-\gamma)\widehat{D}(q-n),\nonumber\\
&C_{q,n,m}=-\zeta^2 \widehat{D}(q-n)\widehat{D}(n-m),
\end{align}
and for $n\neq0$, $n\neq q$,
\begin{align}
\H_{q,n}=\textbf{B}_{q,n}-\widehat{Y}(q-n)\widehat{Y}(n),
\end{align}
then the relation~\eqref{eqnlt2}, together with~\eqref{EDP},~\eqref{PZU} gives (see~\cite{RGKSV} for more details):
\begin{subequations}\label{eqNLP2T}
\begin{align}
(2U_0\widehat{Y}(q)-A_q)P_q=\sum_{n\in\Z}\H_{q,n} P_{q-n}P_n+\sum_{n\in\Z}\sum_{m\in\Z}C_{q,n,m}P_{q-n}P_{n-m}P_m,&\qquad\text{for } q\neq0\label{eqNLP2}\\
U_0^2=u_{00}+\sum_{n\in\Z}\H_{0,n} |P_n|^2+\sum_{n\in\Z}\sum_{m\in\Z}C_{0,n,m}P_{-n}P_{n-m}P_m,&\qquad\text{for } q=0.\label{eqNLP20}
\end{align}
\end{subequations}
Lemma~\ref{lemmetf} justifies this expression as $\widehat{Y}(1-Q_0)$ and $\widehat{D}$ are bounded and satisfy \Hy3.
As said previously, Eq.\eqref{eqNLP20} appears because $\widehat{Y}$ is not defined at 0. In order to obtain an expression similar to~\eqref{eqfreq}, one has to replace $U_0$ in~\eqref{eqNLP2} by its value given in~\eqref{eqNLP20}. For this purpose it is assumed that $U_0$ has a positive value, allowing to express it as a square root. It is assumed that the system encounters a bifurcation for $\gamma_0>0$ which implies $u_{00}>0$. Thus, for a small enough acoustic pressure, this square root can be expanded as a Taylor series. Replacing $U_0$ by this expression in~\eqref{eqNLP2} leads to a nonlinearity of the form presented in lemma~\ref{lemmetf}. As a consequence, theorems~\ref{Thbif} and~\ref{Th1} can be used.
The above relations are complex to write in the canonical form given in~\eqref{eqfreq}, however, for instance one has for the linear part:
\begin{equation}\label{eqJcar}
\widehat{J_{\gamma}}(q)=(2\sqrt{u_{00}}\widehat{Y}(q)-A_q).
\end{equation}
This linear part satisfies \Hy2, provided $(2\sqrt{u_{00}}\widehat{Y}(q)-A_q)=0$ for only two values $q=\pm s$.
Notice that the stationary solutions $U_0=\pm \sqrt{u_{00}}$, where for all $n\in\Z$, $P_n=0$, exist for all $\gamma\in(0,1)$. Consequently, from~\eqref{eqJcar}, the bifurcation parameter and the frequency of oscillation at the bifurcation will be given, according to Th.~\ref{Thbif}, by the characteristic equation:
$$2\sqrt{u_{00}}\widehat{Y}(s)-A_s=0.
$$
The main objective is now attained: hypotheses \Hy2 and \Hy3 hold and if one supposes a solution of the type described in \Hy1, the hypothesis~\eqref{eqWorman} can be justified in the framework of the present study by applying Th.~\ref{Th1}. The results presented in the publications cited in this section, which required at first the assumption~\eqref{eqWorman}, can be obtained by assuming the weaker assumption of twice continuous differentiability of the solution. Let us notice a slight difference between the result of theorem~\ref{Th1} and the relation~\eqref{eqWorman}. The fact that there exists a value $k$ in Th.~\ref{Th1} where~\eqref{eqWorman} stops to be valid does not change the results of the previously cited publications. Indeed, these results can be retrieved by knowing estimates on the first harmonics and involve only a limited number of harmonics in calculus. Moreover, $k$ can be taken as large as needed by limiting the range of amplitude for which the results are valid.
Eventually, the hypotheses required in section~\ref{stable} are satisfied, see~\cite{RGKSV}, and the quantity $D_1(\gamma_0)$ defined in this latter section can be deduced from equations~\eqref{eqNLP2T}. This leads to
$$|P_1|=\sqrt{|\alpha|}\sqrt{\gamma-\gamma_0}+{\cal O}(|\gamma-\gamma_0|^{3/2}),
$$
where 
$$\alpha=\frac{\beta'(\gamma_0)}{D_1(\gamma_0)},
$$
confirming the results of~\cite{RGKSV}.
\section*{Acknowledgements}
This work was supported by the french ANR project CONSONNES. The author wishes to thank J. Kergomard for his great help and advice, B. Lombard for its numerous remarks and suggestions on the manuscript, as well as C. Vergez, P. Guillemain, F. Silva and A. Leger for the discussions and advices about this work.



\end{document}